\documentclass[12pt,a4paper]{article}
\usepackage{amsmath,amsfonts,amsthm,amssymb,latexsym}
\usepackage{xspace}
\usepackage{enumitem}
\usepackage{comment}
\usepackage{graphicx}
\usepackage{mathtools}
\date{}

\newcommand{\Z}{{\mathbb Z}}
\newcommand{\N}{{\mathbb N}}

\newcommand{\emptystring}{\epsilon}

\newcommand{\Aut}{\mathsf{Aut}}

\newcommand{\spine}{\mathsf{spine}}
\newcommand{\coWP}{\mathsf{coWP}}
\newcommand{\WP}{\mathsf{WP}}
\newcommand{\ent}{\mathsf{ent}}
\newcommand{\bs}{\mathfrak{b}}
\newcommand{\ts}{\mathfrak{t}}

\newcommand{\E}{{\rm \sc et0l}\xspace}
\newcommand{\ps}{{\sf p}\xspace}
\newcommand{\ce}{{co-\E}\xspace}
\newcommand{\sce}{{special co-\E}\xspace}
\newcommand{\cspda}{{cspda}\xspace}
\newcommand{\cA}{\mathcal{A}}
\newcommand{\cD}{\mathcal{D}}

\newcommand{\cL}{\mathcal{L}}
\newcommand{\cM}{\mathcal{M}}
\newcommand{\cN}{\mathcal{N}}
\newcommand{\cP}{\mathcal{P}}

\newcommand{\cT}{\mathcal{T}}

\newcommand{\Rev}{\mathsf{R}}
\newcommand{\bp}{\mathbf{p}}
\newcommand{\bq}{\mathbf{q}}
\newcommand{\bx}{\mathbf{x}}
\newcommand{\by}{\mathbf{y}}
\newcommand{\bz}{\mathbf{z}}

\newtheorem{theorem}{Theorem}[section] 
\newtheorem{lemma}[theorem]{Lemma} 

\newenvironment{proofof}[1]{\normalsize {\it Proof of #1}.}{{\hfill $\Box$}}
\newenvironment{mylist}{\begin{list}{}{
\setlength{\parskip}{0mm}
\setlength{\topsep}{2mm}
\setlength{\parsep}{0mm}
\setlength{\itemsep}{0.5mm}
\setlength{\labelwidth}{7mm}
\setlength{\labelsep}{3mm}
\setlength{\itemindent}{0mm}
\setlength{\leftmargin}{12mm}
\setlength{\listparindent}{6mm}
}}{\end{list}}

\title{Groups with \E co-word problem}
\author{Raad Al Kohli, Derek F. Holt, Sarah Rees}
\begin{document}
\maketitle

\begin{abstract}
We study groups whose co-word problems are \E languages, which we call
co-\E groups, using an automaton based model due to van Leeuwen, and recently
studied by Bishop and Elder.  
In particular we prove a number of closure results for the class of groups with
co-word problems in a subclass of `special' \E languages; that class of groups
contains all groups that we know at the time of writing to be co-\E, including
all groups that were proved by Holt and R\"over to be stack groups, and hence
co-indexed.  It includes virtually free groups, bounded automata groups,
and the Higman-Thompson groups, together with groups constructed from those
using finitely generated subgroups, finite extensions, free and direct products,
and by taking the restricted standard wreath product of a co-\E group by a
finitely generated virtually free top group.

\end{abstract}
2010 AMS Subject classification: 20F10, 68Q42, 68Q45.

Keywords: ET0L language, check stack pushdown automaton, indexed language,
word problem, co-word problem.

\section{Introduction}
A finitely generated group is said to be \emph{co-\E} if the complement of
its word problem (i.e.\ its co-word problem) lies in the class of \E languages.
Our aim in this article is to prove that certain types of groups are co-\E, and
to establish some closure results for this class of groups.
More precisely we prove these results for the class of groups with co-word
problem that is a \emph{special} \E language; that class of groups contains all
groups that we know at the time of writing to be co-\E, and it coincides with
the class of all groups that are proved in~\cite{HR} to be co-indexed.
In particular, it includes virtually free groups, bounded automata groups,
and the Higman-Thompson groups, together with groups constructed from those
using finitely generated subgroups, finite extensions, free and direct products,
and by taking the restricted standard wreath product of a co-\E group by a
finitely generated virtually free top group.

Various classes of formal languages of higher complexity than that of
context-free languages have been 
examined as candidates to define natural complexity classes for certain
problems in group theory, since Muller and Schupp's seminal result~\cite{MS}
characterising finitely generated virtually free groups via their context-free
word problems.  Indexed languages, corresponding to indexed grammars and nested
stack automata, entered the arena with Bridson and Gilman's result~\cite{BG}
finding combings for the fundamental groups of all compact 3-manifolds, but the
subclass (within the indexed languages) of \E languages started to be 
studied after results of Ciobanu, Diekert and Elder in~\cite{CDE,CE} showed that
such languages contained solution sets to equations over first free and then
hyperbolic groups.  Holt and R\"over had proved in~\cite{HR} that all
finitely generated bounded automata groups (and so in particular Grigorchuk's 
group of intermediate growth and the Gupta-Sidki group) are co-indexed
(i.e.\ they have co-word problems that are indexed languages), as well as the
Higman-Thompson groups (later proved in~\cite{LS} to be co-context free). At the
same time they had proved various closure properties for a subset of the family
of co-indexed groups, whose members they called stack groups and which contained
all groups that they could prove to be co-indexed. Ciobanu, Elder and Ferov~\cite[Theorem 6.2]{CEF}
then proved that the Grigorchuk group is in fact co-\E. Subsequently Bishop
and Elder~\cite{BE} derived the same result for all finitely generated bounded
automata groups, using an automata-based characterisation of \E languages.  

In this article, motivated by the characterisation of \E languages that was
used in~\cite{BE} (and attributed to van Leeuwen~\cite{vL}) and by closure
results for specific co-\E groups appearing in the thesis work of
Al Kohli~\cite{AlKohli} (it is proved there that the free 
products $\Z^m * \Z^n$ and $\Z^m * G$, for $G$ virtually free, are co-\E), 
we sought to convert the results of~\cite{HR} from indexed to a restricted
class of \E languages. It had been suggested in~\cite{AlKohli}
that the results of~\cite{HR} would carry over to the class
of groups which we shall refer to here as \emph{special} co-\E groups.
Specifically, we have converted the restrictions on nested stack automata that
were imposed on those recognising the co-word problems of the stack groups
of~\cite{HR}, in order to define the classes of special \E languages,
and of special co-\E groups.

We prove in Theorems~\ref{prop:vfree},~\ref{prop:bag}, and~\ref{prop:htg}
that the class of special co-\E groups contains all virtually free groups, all
finitely generated bounded automata groups and the Higman--Thompson groups,
extending to \sce the results of~\cite{MS,BE,HR}.
In Theorem~\ref{prop:clfin} we prove closure of the class of \sce groups
under passing to finitely generated subgroups, finite direct products and
finite extensions, and that the property of being special co-\E is independent
of finite generating set. 
In Theorems~\ref{prop:clwp} and~\ref{prop:clfp} we prove closure under
taking the restricted wreath product with a finitely generated virtually
free group, and under taking free products with finitely many free factors.

We are grateful to Murray Elder for some very helpful conversations about the
\cspda approach to \E languages, and in particular for the formulation of
Property (P6) in our definition of a special \cspda below.
We are also grateful to the referee for several helpful suggestions,
including a clearer enumeration of the assumed properties of a special \cspda,
which have enabled us to improve the clarity of the proofs.

\section{Definitions and assumptions}
Our definition of a \emph{check stack pushdown automaton}
(which we shall abbreviate as \cspda) is basically taken from~\cite{BE}.
Automata with the same name were also introduced by van Leeuven in~\cite{vL},
with a slightly different definition.  
However, it is proved in~\cite{vL}, and also in an Appendix to~\cite{BE} that
is attached to the unpublished version~\cite{BE2}, that
the class of languages that are recognisable by the \cspda defined in the
respective article is precisely the class of \E languages. So these definitions
are equivalent, and we will use the one from~\cite{BE}.

We define a \cspda $\cM$  to be a non-deterministic finite-state automaton with a
one-way input tape and a read write head, which has access to both a check
stack (with stack alphabet $\Delta$) and a pushdown stack (with stack alphabet
$\Gamma$), where access to these two stacks is tied in a very particular way,
as we explain below.
Throughout this article we restrict attention to \cspda 
in which a number of extra properties, (P1)--(P7), hold. 
We call such a \cspda
a \emph{special} \cspda, and the language accepted by such a \cspda a
\emph{special} \E language. 
We define the properties (P1)--(P7) below.

The execution of any \cspda (special or otherwise) is separated into two stages.

In the first stage the machine is allowed to push to its check stack but not
to its pushdown stack, and further, the machine is not allowed to read from its
input tape.  Since the words over $\Delta$ that can be pushed to the check
stack are controlled by a finite state automaton, the language consisting of
all such words is regular.
We shall refer to this first stage as the \emph{initialisation} of the \cspda.

In the second stage, 
the machine can read from the input tape and from its check stack 
(but no longer alter the check stack), and can push to and pop from the pushdown stack.
Movement along the check stack is tied to movement along the pushdown stack, so that every
time the machine pushes a value onto the pushdown stack its read write head
moves up the check stack, and every time the machine pops a value off the pushdown
stack the head moves down the check stack.  
So, throughout the second stage of
the computation, the read write head can see the top symbol of the pushdown
stack and the symbol in the corresponding position on the check stack 
(if the head is below the top of the check stack).

A \cspda  accepts its input word if it enters an accepting state at some
stage after reading all of its input.

By definition, a special \cspda must possess properties (P1)--(P7) below.
As we shall see, some of these properties can
be achieved by modifying the original \cspda without changing its language,
but others, including (P2), (P4) and (P6), are likely to impose restrictions on
the class of accepted languages.  We shall now list and discuss these assumed
properties.

\begin{enumerate}
\item[(P1)] Every state of the automaton is of one of the following two
types:
\begin{itemize}
\item every out transition from that state is reading, in which case we say
that the state is \emph{reading}; or
\item every out transition from that state is non-reading, in which case we
say that the state is \emph{non-reading}.
\end{itemize}

It is clear that, by introducing new states if necessary, we can modify any
\cspda to define one with the same language that satisfies this property.

\item[(P2)] The automaton is \emph{deterministic upon input}. More precisely,
the operation of the \cspda throughout the second stage is deterministic.

\item[(P3)] 
The check and pushdown stack alphabets $\Delta$ and $\Gamma$
both contain a single \emph{bottom of stack} symbol and single \emph{padding} symbol, which we call $\bs$ and $\ps$
(or, if necessary, $\bs_{\cM}$ and $\ps_{\cM}$,  or $\bs_G$ and $\ps_G$, for a group $G$),   
and $\Delta$ contains a single \emph{top of stack} symbol $\ts$ (or $\ts_{\cM}$ or $\ts_G$). After the
initialisation, the check stack has $\bs$ at the bottom and $ \ts$ at the top.
The first transition of the second stage is to write $\bs$  at the bottom of
 the pushdown
stack, after which the automaton moves into a reading state.
We assume further that:
\begin{itemize}
\item the symbol $\bs$ occurs only at the bottom of the two stacks,
and
$\ts$ occurs only at the top of the check stack;
\item the symbol $\bs$ cannot be popped;
\item when the symbol $\ts$ is read from the check stack, no symbol can be pushed onto
 the pushdown stack, that is, the read write head cannot move above the top of the check stack.
\end{itemize}
This property can be achieved by modifying an arbitrary \cspda without
changing its accepted language. 
In order to achieve the last of the three extra conditions, an arbitrary number of padding symbols $\ps$ are written at the top of the check stack during the non-deterministic initialisation stage, before $\ts$ is written; whether or
not a word is accepted in a particular computation may depend on the length of that string
of padding symbols. The other modifications are straightforward.

\item[(P4)] It is not possible for there to be an infinitely-long sequence of
non-reading transitions in any computation of the automaton.

\item[(P5)] All accepting states of the automaton are reading states.

Any \cspda that satisfies (P4) can be modified to satisfy (P5) without 
changing the accepted language.
\end{enumerate}

Suppose now that our \cspda, $\cM$, possesses properties (P1)--(P5).
Property (P4) allows us to define a configuration of 
$\cM$ to be
\begin{itemize}
\item \emph{reading} if it is at a reading state;
\item \emph{accepting} if it is at an accepting state
(which is also reading, by (P5)); and
\item \emph{failing} if it is a non-accepting reading configuration from which it is not
possible to reach an accepting configuration.
\end{itemize}

We shall use the term \emph{entry configuration} to denote a configuration
that arises immediately after the
first move of the second stage in which we write $\bs$ to the pushdown stack,
and we shall call its associated state an \emph{entry state}.
Note that these are reading configurations by (P3).
By contrast a state and configuration of $\cM$ 
at the beginning of
stage 1 are called an \emph{initial state} and \emph{initial configuration}.
We shall assume that 
$\cM$ has a single initial state.

Given a reading configuration $C$ of the automaton, and a word $w$ over the
input alphabet, we define $C^w$ to be the machine's configuration as it enters
its first reading state after reading all of $w$, or a
\emph{failing configuration} if it enters a failing configuration
without reaching such a reading state. In the second case we say that
$C^w$ \emph{fails}. Note that $C^w$ is well-defined, by (P2) and (P4).

The next assumed property is that 
$\cM$ \emph{ignores} subwords
of input words that are not in its language $\cL(\cM)$, or more precisely:

\begin{enumerate}
\item[(P6)]
Let $C$ be an entry configuration of $\cM$, and $u,v \in A^*$ with
$v \notin \cL(\cM)$. Then either $C^{uv} = C^u$, or $C^{uv}$ fails;
In particular (under the same assumptions), either $C^v=C$ or $C^v$ fails.
\end{enumerate}

The final statement of (P6) is derived from the previous one by setting $u=\emptystring$.
It follows from that final statement 
that if $D$ is any reading non-failing non-entry configuration, then for any
entry configuration $C$ and any $v$ for which $C^v=D$, we have $v \in \cL(\cM)$.
Hence whether or not $D$ is an accepting configuration does not affect the
language $\cL(\cM)$,
and hence, when (P1)--(P6) hold we may assume that every such configuration $D$ is accepting, 
and hence that all non-failing non-entry reading states of $\cM$ are accepting. 
We note also  that if a state of $\cM$ is accepting, then by (P5) it is reading, and by definition it is non-failing and non-entry.

We have shown that if (P1)--(P6) hold for $\cM$ then the following property may be 
assumed (that is, that $\cM$ may be modified to a machine accepting the same language in which it holds). But we find it 
convenient to identify it as a separate property.

\begin{enumerate}
\item[(P7)] The set of accepting states of $\cM$ is equal to the set of reading
states that are not entry states.
\end{enumerate}

In this article, we are interested in \cspda accepting co-word problems of
finitely generated groups $G$. So we shall assume from now on that the input
alphabet of a \cspda under consideration consists of a (finite) inverse-closed 
generating set $A$ of such a group $G$, and that $A \subseteq \Delta$ and
$A \subseteq \Gamma$.  We denote the word and co-word problems of $G$ with
respect to $A$ by $\WP(G,A) \coloneqq \{w \in A^* \mid w =_G 1 \}$ and
$\coWP(G,A)\coloneqq \{w \in A^* \mid w \ne_G 1 \}$, respectively.

We call a finitely generated group $G$ a \emph{\sce group} if, for some finite
inverse-closed generating set $A$ of $G$, there is a special \cspda with
input alphabet $A$ that accepts the language $\coWP(G,A)$.
We shall prove in Theorem~\ref{prop:clfin} below that the property 
of being a \sce group does not depend on the choice of the generating set $A$.
We note that Property (P6), which is required for the \cspda  when $G$ is a
\sce group, is precisely what was referred
to as \emph{ignoring the word problem} of $G$ in~\cite{HR}.
We also note that every \sce group is a stack group (in the sense of~\cite{HR}).

The following lemma is useful for avoiding the failure of $C^{uv}$ in the
situation described in Property (P6). It will be used in the proof of
closure of the class of \sce groups under free products in
Theorem~\ref{prop:clfp} below.

\begin{lemma}\label{lem:nonfail}
Let $G = \langle A \rangle$ be a group, $\cM$  a special \cspda with
$\cL(\cM) = \coWP(G,A)$, $w \in \cL(\cM)$, and $n \in \Z$ with $n>0$.
Then there is an entry configuration $C$ of $\cM$ for which 
$C^w$ is accepting, and such that $C^v$ does not fail for any word $v$ over $A$
of length at most $n$.
\end{lemma}
\begin{proof}
Let $v_1,v_2,\ldots,v_N$ be an enumeration of all words over $A$ of length at
most $n$. Then the word 
$w'\coloneqq v_1v_1^{-1}v_2v_2^{-1}\cdots v_Nv_N^{-1}w$ lies in
$\coWP(G,A)$, and so there is an entry configuration $C$
such that the state of $C^{w'}$ is accepting.
It follows that for any prefix $w''$ of $w'$, $C^{w''}$ does not fail.

For any $i$ with $0 \le i < N$, it follows from the second statement of Property (P6) 
of a special \cspda (applied with $v=v_1v_1^{-1} \cdots v_iv_i^{-1}$)
that $C^{v_1v_1^{-1}\cdots v_iv_i^{-1}}=C$
and hence that $C^{v_1v_1^{-1}\cdots v_iv_i^{-1}v_{i+1}}=C^{v_{i+1}}$.
So $C^{v_{i+1}}$ does not fail.

Similarly, we deduce from Property (P6) that 
$C^{v_1v_1^{-1}\cdots v_Nv_N^{-1}}=C$ and hence $C^{w'}=C^{w}$.
So $C^{w}$ is accepting, as required.
\end{proof}

\section{Virtually free groups}\label{sec:vfree}
Theorem~\ref{prop:vfree}, which we prove in this section, is the special case
of Theorem~\ref{prop:clwp} in which the group $H$ is trivial. But we prefer to
prove Theorem~\ref{prop:vfree} first, firstly because it is more
straightforward, and secondly because the technique used in the proof
is applied again in the proof of Theorem~\ref{prop:clwp}.
We observe in addition that, since the Thompson group $V$ is known to contain
all virtually free groups as subgroups (see, for example,~\cite[Lemma 1.1,
Theorem 1.4, Corollary 1.3]{BSD}), this result could also be deduced from
Theorems~\ref{prop:htg} and~\ref{prop:clfin}.

We note that in the proof of Theorem~\ref{prop:vfree}, and also in the proofs
of Theorems~\ref{prop:bag} and~\ref{prop:htg}, we refer forwards to and
use Theorem~\ref{prop:clfin}. The proof of that result is elementary, and does
not depend on earlier results in this article, and we have chosen to defer its
proof to Section~\ref{sec:clprops}, which is devoted to closure properties
of the class of \sce groups.

\begin{theorem}\label{prop:vfree}
Every finitely generated virtually free group is a \sce group.
\end{theorem}
We note that the fact that finitely generated virtually free groups are
\ce groups follows immediately from the inclusion of the context-free languages
in \E (proved in~\cite[Theorems 2.4,2.7]{KRS}) and the fact that virtually
free groups have deterministic context-free word problem, and hence
context-free co-word problem.
But here we  prove additionally that these groups are \sce groups, by constructing a special \cspda.
\begin{proof}
Let $G$ be finitely generated and virtually free.
We shall prove in Theorem~\ref{prop:clfin} below that the property of
being a \sce group does not depend on the choice of the finite generating set
of the group $G$.
We choose a generating set $A= B \cup T^{\pm 1}$ of $G$ in which
$B$ is the inverse closure of a free generating set of a free subgroup $F$ of
finite index in $G$ and $T$ is a right transversal of $F$ in $G$ with
the element in $F$ removed. Then every element of $G$ can be written
uniquely as a word $wt$ with $w$ a freely reduced word in $B^*$ and
$t \in T \cup \{\emptystring\}$, where $\emptystring$ denotes the
empty word. It was shown in~\cite{MS} that the word problem $\WP(G,A)$ is the
language of a deterministic pushdown automaton with states corresponding to
$T \cup \{\emptystring\}$, that stores the freely reduced word $w$ on its stack.
We can amend that to a pushdown automaton $\cP_G$ that accepts when its stack
is empty, by storing $t$ (if non-empty) on the stack above $w$.

We convert $\cP_G$ to a \cspda by attaching a check stack onto which we write
a string $\bs \ps^k \ts$ at the start, for some arbitrary integer $k>0$.
It has just two states, the entry state, which is entered if and only if
$\bs$ is seen on the two stacks, and an accepting state.
After the initialisation and writing $\bs$ on the pushdown stack, the \cspda
operates as $\cP_G$, additionally reporting failure if  
(P3) is violated by an attempt to push a symbol onto the pushdown stack above $\ts$.
The new automaton accepts any word that $\cP_G$ rejects, provided that $k$ has been chosen to be large enough,
and rejects if $\cP_G$ accepts (or fails if $k$ is too small).

It is straightforward to verify that this \cspda is special with language
$\coWP(G,A)$.
Note that failure can only arise as a result of inadequate initialisation
of the check stack, and this can be remedied without otherwise affecting its
operation by pushing more \ps symbols onto the check stack.
\end{proof}

\section{Bounded automata groups}\label{sec:bag}
This section is devoted to the proof of the following result, which strengthens
the main result of~\cite{BE}; its proof is based on the proof of that result and
the proof of the corresponding result in~\cite{HR} on which it is based.

\begin{theorem}\label{prop:bag}
Every finitely generated bounded automata group is a \sce group.
\end{theorem}

Bounded automata groups were introduced by Sidki in~\cite{Sidki}.
The best known examples are the Grigorchuk group and the Gupta--Sidki groups,
and more recently studied examples include the basilica groups~\cite{CZ}.
Our definition below describes them as groups of permutations of sets of
infinite strings, which are generated by finite state transducers
(i.e.\ automata with input and output).
Finitely generated bounded automata groups were proved to be
co-indexed in~\cite{HR}, and then to be \ce in~\cite{BE}.

For $d \geq 2$, we consider the set $\Sigma^*$ of strings over the finite set
$\Sigma$ of $d$ elements, and write $\epsilon$ for the empty string. 
There is a natural bijection between $\Sigma^*$ and the set of vertices of the
$d$-regular rooted tree $\cT_\Sigma$, in which the $d$ edges down from any
given vertex are labelled by the elements of $\Sigma$; similarly there is a
natural bijection between $\Sigma^\omega$ and the set of ends of $\cT_\Sigma$.
We consider the group of all permutations of $\Sigma^*$ that induce
automorphisms of $\cT_\Sigma$, which we shall denote here by $\cA_\Sigma$
(it is denoted by $\Aut(\cT_d)$ in~\cite{BE}); such a permutation $\alpha$ must
preserve `levels' in the tree, and hence for any $k$ it permutes the set of subsets of
strings with common prefixes of length $k$; that is the subsets
$\bx \Sigma^*$ of $\Sigma^*$ with $\bx \in \Sigma^k$.  In addition it must have
the property that, for any $\bx \in \Sigma^*$ and,
$\by \in \Sigma^* \cup \Sigma^\omega$, the string $(\bx\by)^\alpha$ has
$\bx^\alpha$ as a prefix, and the restriction $\alpha|_\bx$ of $\alpha$ to
$\bx$ is defined by $(\bx\by)^\alpha = \bx^\alpha \by^{\alpha|_\bx}$.

An automorphism $\alpha \in \cA_\Sigma$ is called an
\emph{automaton automorphism} if there is an associated finite state transducer
$\cN_\alpha$ such that for any $\bx\in \Sigma^*$ the image $\bx^\alpha$ of
$\bx$ is equal to the output string $\bx^{\cN_\alpha}$ of the transducer when
$\bx$ is input.
An automorphism is called \emph{bounded} if for some $N \in \N$ and every
$k \in \N$ there are at most $N$ strings $\bx \in \Sigma^k$ with
$\alpha|_\bx \ne 1$.  The set of all bounded automaton automorphisms forms a
subgroup of $\cA_\Sigma$, which we shall call $\cD_\Sigma$.
A \emph{bounded automata group} is defined to be
a subgroup of $\cD_\Sigma$, for some finite set $\Sigma$.

An automorphism $\phi\in \cA_\Sigma$ is called \emph{finitary} if, for some
$k_\phi \in \N$ and all $\bx$ with $|\bx|\geq k_\phi$, the restriction
$\phi|_\bx$ is equal to the identity.

We refer the reader to~\cite{BE} for the precise definition of a \emph{directed
automaton automorphism}; such an automorphism $\delta$ has an associated
\emph{spine}, an element of $\Sigma^\omega$, with the property that
$\delta|_\by$ is finitary for any $\by$ that is not a prefix of
$\spine(\delta)$. 
It is proved in~\cite[Lemma 1]{BE} that the spine of any directed automaton
automorphism is eventually periodic, that is, of the form $\bp\bq^\omega$,
for fixed strings $\bp,\bq \in \Sigma^*$,
and that $(\bp\bq^\omega)^\delta = \bp' \bq'^\omega$, for fixed strings
$\bp',\bq' \in \Sigma^*$ that have the same lengths as $\bp$ and $\bq$.
Furthermore, $\bp$ and $\bq$ can be chosen such that,
if $\bq=b_1\cdots b_t$ with $b_i \in \Sigma$, then
$\delta|_{\bp \bq^k b_1 \cdots b_j} = \delta|_{\bp  b_1 \cdots b_j}$ for
any $k \ge 0$ and $0 \le j < t$.

Our proof of Theorem~\ref{prop:bag} (below) is a modification of the proof
of~\cite{BE}, which itself emulates the proof in~\cite{HR}.
We shall describe a construction that is
basically the construction of~\cite{BE}, with some modification to fit the
constraints that we have imposed on a \cspda to make it special.
We shall not repeat proofs of properties of the \cspda that are already proved 
in~\cite{BE}.

We note that in~\cite{BE} bounded automata groups are described as groups of 
automorphisms of rooted trees, but we prefer to view them as described above,
in terms of their associated actions on sets of finite strings as well as the
sets of infinite strings that correspond to the ends of the trees.
We also note that the entry states in our \cspda are analogous to states of the
nested stack automaton constructed in~\cite{HR} that are called
{\em start states} in that article. 

\begin{proofof}{Theorem~\ref{prop:bag}}
Let $G$ be a subgroup of $\cD_\Sigma$ with finite generating set $X$.
It is proved in~\cite[Proposition~16]{HR} that $\cD_\Sigma$ is generated by 
the infinite subset of all of its finitary and directed automaton automorphisms. 
Hence the subgroup $G$ of $\cD_\Sigma$ 
is a finitely generated subgroup of a group 
with finite generating set $Y$, each of whose elements is either finitary
or a directed automaton automorphism. 
It will follow from
Theorem~\ref{prop:clfin} below that $G$ is a \sce group
provided that the subgroup $\langle Y \rangle$ of $\cD_\Sigma$ is \sce.
Hence, as in~\cite{BE}, we may assume that $G$ itself has such a
generating set $Y$.

The proof in~\cite{BE} that $G$ is \ce defines a \cspda that provides a 3-stage
process to verify non-deterministically that a given element $\alpha \in G$ is
non-trivial.  In Stage 1, an arbitrary string $\bx=x_1\cdots x_m \in \Sigma^*$
is put onto the check stack, with $x_1$ as top symbol. The image $\bx^\alpha$ is
computed in Stage 2, which is then compared with $\bx$ in Stage 3 and, 
if $\bx \neq \bx^\alpha$, then $\alpha$ is proved non-trivial.
If $\alpha$ is non-trivial, then there must be some such word $\bx$ with
$\bx \ne \bx^\alpha$, and so the language of the \cspda is equal to
$\coWP(G,Y)$, as required.

Stage 1 of the \cspda that we define here is essentially the same as
in~\cite{BE}, with $x_1$ as the top symbol on the check stack, but to
satisfy Property (P3) we also write $\ts$ on the check stack above $x_1$.
In the construction of~\cite{BE}, $\bx$ is put immediately onto the pushdown
stack at the end of Stage 1 but, in our modification, to ensure that (P3)
is satisfied, we defer copying the contents of the check stack onto the
pushdown stack until after we have entered an entry configuration of the
\cspda and then have read the first letter in the input word.

During the second stage of the process described in~\cite{BE}, the contents of
the pushdown stack are modified so that it contains $\bx^\alpha$ by the end of
that stage; successive images of $\bx$ under the generators in the string
$\alpha$ are computed, as explained below.
During the third stage of that process, the contents of the two stacks are
compared with each other, to see if they differ.

Our Stage 2 process is similar to that of~\cite{BE}, but it does a little more
than in~\cite{BE}. For, in order to maintain determinism after Stage 1,
we do not have a Stage 3,
but do our checking within Stage 2; the checking will be explained later.
In our Stage 2, we have just two reading states. The first is the unique
entry state and is non-accepting, and the second is an accepting state;
we shall call those two states $q_1$ and $q_2$, respectively.
As required by (P5), all non-reading states are non-accepting.

In our description of Stage 2,
we denote the current image $\bx^\beta$ of $\bx$ under the prefix $\beta$ of the
input word that has been read so far by $\by=y_1y_2 \cdots y_m$. We need to
explain how we compute the image of $\by$ under finitary and directed
automorphisms in Stages 2a and 2b, respectively.

Stage 2a of the process described in~\cite{BE} computes the
image of $\by$ under a finitary element.
Since this involves changing only a bounded number of symbols at the top of
the pushdown stack, this is easily done deterministically using finite memory,
and we shall not give the details here; they can be found in~\cite{BE}.

Stage 2b of the process described in~\cite{BE} computes the image of $\by$
under a directed automaton automorphism $\delta$.
Let $\bz =y_1\cdots y_\ell$ 
be the maximal prefix of $\by$ that is also a prefix of $\spine(\delta)$.
We shall assume that $\ell<m$; the case $\ell=m$ (where $\bz=\by$) is an easy
modification of the case $\ell<m$.  As we saw above, we have
$(\bp\bq^\omega)^\delta = \bp' \bq'^\omega$, for fixed strings
$\bp=a_1\cdots a_s,\bq=b_1\cdots b_t$,
$\bp'=a'_1\cdots a'_s,\bq'=b'_1\cdots b'_t \in \Sigma^*$.

Then, either $\ell \le s$, in which case $\bz=a_1\cdots a_\ell$ and
\[\by^\delta= a'_1\cdots a'_{\ell}(y_{\ell+1})^{\delta|_\bz}(y_{\ell+2}
	\cdots y_{\ell_\phi})^\phi y_{\ell_\phi+1} \cdots y_m; \]
or $\ell > s$ and, for some $k \ge 0$ and $j$ with $0 \le j < t$, we have
$\ell=s+kt+j$, $\bz=\bp\bq^kb_1\cdots b_j$ and
\[\by^\delta = \bp' \bq'^kb'_1\cdots b'_j
	(y_{\ell+1})^{\delta|_\bz}(y_{\ell+2}\cdots y_{\ell_\phi})^\phi
  y_{\ell_\phi+1} \cdots y_m, \]
where in both cases $\phi$ is the finitary automorphism
$\delta|_{\bz y_{\ell+1}}$ with associated bound $k_\phi$ on the length
of the prefix of words that it moves,
and $\ell_\phi \coloneqq \min(\ell+1+k_\phi,m)$.

We need to pop the prefix of $\by$ of length $\ell_\phi$ from the stack.
Then we compute its image under $\delta$ and then push that (or rather its
reverse, which we denote by a superscript $\Rev$) onto the stack.

The case where $\ell\leq s$ is straightforward. The computation and storage of  
$y_{\ell+1}^{\delta_\bz}$ and $(y_{\ell+2} \cdots y_{\ell_\phi})^\phi$
need a bounded amount of memory, and so can be 
done deterministically using non-reading states and transitions of the \cspda.
Then we push
$((y_{\ell +2}\cdots y_{\ell_\phi})^\phi)^\Rev$ 
onto the pushdown stack followed by
$(y_{\ell+1})^{\delta|_\bz}a'_\ell\cdots a'_1$.

The case where $\ell\geq s$ is more complicated.
As explained above, it is proved in~\cite[Lemma 1]{BE} that
$\delta|_\bz = \delta|_{\bp \bq^k b_1 \cdots b_j}$ does not depend on $k$,
and so $\delta|_\bz$ and the finitary automorphism $\phi$
depend only on $\bp$, $\bq$ and $j$. So the computation and storage of the
components of $\by^\delta$, namely $\bp'$, $\bq'$, 
$y_{\ell+1}^{\delta_\bz}$ and $(y_{\ell+2} \cdots y_{\ell_\phi})^\phi$,
is again straightforward, and needs only a bounded amount of memory. 
As above, we can then push
$((y_{\ell +2}\cdots y_{\ell_\phi})^\phi)^\Rev$ 
onto the pushdown stack followed by
$(y_{\ell+1})^{\delta|_\bz}b'_j\cdots b'_1$
But then we need to push $k$ copies of $\bq'^\Rev$ followed by $\bp'^\Rev$
without knowing the value of $k$ in advance, and $k$ could take any non-negative
integer value. This computation cannot be achieved using bounded memory.

In~\cite{BE}, this is achieved non-deterministically by guessing the value
of $k$, checking whether the heights of the two stacks are the same at the
end of the process, and failing if they are not.

But we can achieve this deterministically by repeatedly pushing $\bq'^\Rev$
until there are fewer than $s+t$ cells between the
top of the pushdown stack and the top of the check stack (which can be tested
by making the read write head move up while pushing padding symbols onto the
pushdown stack), and at that stage pushing $\bp'^\Rev$.  At the end of this
deterministic process we have $\by^\delta$ on the pushdown stack.

In the process described in~\cite{BE}, the \cspda decides after processing each
input symbol either to read another input symbol, or to proceed to Stage 3
and look for no further input. If the machine moves to 
Stage 3 prematurely, then the end of the input word is never reached, and 
so that word is not accepted.
If the machine looks for input after reading the end of the word, then it must also fail. The machine can only accept the word if it makes that transition at
the right moment and finds that the contents of the two stacks are distinct.
Clearly this part of the process of~\cite{BE} is non-deterministic.

In our \cspda, in order to maintain determinism, we compare the contents of 
the two stacks during our Stage 2 repeatedly,
after we have processed each input symbol. We do this by popping symbols off 
the pushdown stack until we reach a point at which the two stacks differ; 
the first symbol of the pushdown stack that is not matched is then pushed back 
onto the stack.  If the two stacks are distinct, then we move into the single
accepting reading state $q_2$, but if they match we return to the single
non-accepting reading (and  entry) state $q_1$.
In the event of no further input that state will be the final state.  If there
is further input then we read it, and then we copy across from the check stack
to the pushdown stack to fill those cells at the top of the pushdown stack
that we have just popped, so that both stacks have the same height as before,
and contain the same contents as they did just before this checking process.
We then proceed to process the new input symbol as already described above.

It is proved in~\cite{BE} that the machine described there is a \cspda that 
accepts an input word precisely when it is in $\coWP(G,A)$.
It should be clear that our modified \cspda accepts the same language as does 
the \cspda of~\cite{BE}, and that our modifications ensure that properties
(P1)--(P5) and (P7) hold for our \cspda.

It remains to consider Property (P6). 
Since our machine never fails, we only need to check (P6)\,(a),
that for any words $u \in A^*, v \in \WP(G,A)$ the configurations 
$C^u$ and $C^{uv}$ are equal. 
These configurations are determined by the state of the machine
and the contents of the pushdown stack.
At any stage the pushdown stack must contain the image of the string on the 
check stack under the action of the element represented by the input
string that has been read so far.
Since $u,uv$ represent the same group element, the pushdown stack contents in the two configurations must match, as required.
And the state within each configuration must be either $q_1$ or $q_2$,
depending on whether or not that single element is equal to the identity.
\end{proofof}

\section{Higman--Thompson groups}\label{sec:htg}

The Higman--Thompson groups $G_{n,r}$ were described by Higman
in~\cite{Higman} as generalisations of one of Thompson's finitely presented
infinite simple groups, namely the group
$V$ described in~\cite{Thompson}, which is isomorphic to $G_{2,1}$, and which
first appeared in a handwritten manuscript of Richard J. Thompson circa 1965.
They were proved to be co-indexed in~\cite{HR} and the stronger result
that they are co-context-free is proved in~\cite{LS}.
In this section we extend the proof of~\cite{HR} to derive the following result.

\begin{theorem}\label{prop:htg}
The Higman--Thompson groups $G_{n,r}$ are \sce groups.
\end{theorem}

Note that, as finitely generated subgroups of the group $V \cong G_{2,1}$,
the Thompson groups $F$ and $T$ are proved special \ce by the combination of
Theorem~\ref{prop:htg} and Theorem~\ref{prop:clfin} below.

There are several equivalent definitions of the groups $G_{n,r}$, and the one
we use here is taken from~\cite[Section~1.10.11]{HRR}, and is also used
in~\cite[Section~3]{HR}.  Our proof that these groups are \sce is based on the
proof in~\cite[Theorem~1]{HR} that they are stack groups, and is similar to our
proof of Theorem~\ref{prop:bag}.

Fix integers $r \ge 1$ and $n \ge 2$, and let
$Q \coloneqq \{q_1, q_2,\ldots,q_r\}$ and $\Sigma \coloneqq
\{\sigma_1, \sigma_ 2, \ldots, \sigma_n\}$ be disjoint sets of
cardinality $r$ and $n$ respectively. Define $\Omega$ to be the set of all
infinite strings of the form
$\omega = q_i \sigma_{i_1} \sigma_{i_2} \sigma_{i_3} \cdots$,
where $q_i \in Q$ and $\sigma_{i_j} \in \Sigma$ for $j \ge 1$.
A non-empty prefix of $\omega$ is a finite string (or word)
$q_i \sigma_{i_1} \cdots\sigma_{i_k}$ for some $k \ge 0$.  Roughly speaking
$G_{n,r}$ is the group which acts on $\Omega$ by replacing non-empty prefixes;
we give more detail below.

Observe that $Q \Sigma^*$ is the set of non-empty prefixes of strings in
$\Omega$.
We define a \emph{finite complete antichain} in $\Omega$ to be a finite subset
$B$ of $Q \Sigma^*$ so that each $\omega \in \Omega$ has precisely one element
of $B$ as prefix. For example, when $r=1$ and $n=2$,
$\{q_1 \sigma_1, q_1 \sigma_2 \sigma_1, q_1 \sigma_2 \sigma_2\}$ is a finite
complete antichain, 
but $\{q_1 \sigma_1,\, q_1 \sigma_1 \sigma_1,\,\allowbreak
q_1 \sigma_2 \sigma_1,\,\allowbreak q_1 \sigma_2 \sigma_2 \}$ is not.

A bijection $\varphi \colon B \to C$ between two finite complete antichains
induces a permutation of $\Omega$, also called $\varphi$, defined by
$\omega^\varphi \coloneqq b^\varphi w'$, where $\omega = b\omega'$ and $b$
is the unique prefix of $\omega$ that lies in $B$.  The group $G_{n,r}$ is
defined to be the group of all such permutations of $\Omega$. We also define
$\bx^\varphi \coloneqq b^\varphi \omega'$ for any prefix $\bx = b \omega'$ of
an element of $\Omega$ that has $b \in B$ as a prefix.

\begin{proofof}{Theorem~\ref{prop:htg}}
This proof is very similar to that of Theorem~\ref{prop:bag} and the
computation proceeds by making prefix replacements to a word stored on
the pushdown stack. The main difference is that prefixes may be replaced by
words of a different length. As in the earlier proof, there are just two
reading states $q_1$ and $q_2$, which are respectively non-accepting and
accepting, where the state is $q_1$ when the pushdown stack contains only $\bs$,
and $q_2$ otherwise. All non-reading states are non-accepting (as required by
(P5)).

It is shown in~\cite{Higman} that $G_{n,r}$ is finitely presented. So we fix a
finite inverse-closed set generating $A$ of $G_{n,r}$. Then there is a bound $K$
such that all finite complete antichains $B$ and $C$ occurring in the elements
$(B, C, \varphi)$ of $A$ have length at most $K$. 

In its first stage, the \cspda with language $\coWP(G,A)$ puts an arbitrary
word $\bx$ in $Q\Sigma^*$ on the check stack, with its first symbol $q_i$ at
the top, and then puts an arbitrary number of padding symbols on top of that
followed by $\ts$.
The read write head then moves to the bottom of the stacks, and the \cspda
enters its deterministic second stage, puts $\bs$ on its pushdown stack, and
goes into its non-accepting reading state $q_1$.
After the first letter (if any) of the input word $\alpha$ has been read, 
the word $\bx$ is copied onto the pushdown
stack. Thereafter, the \cspda works by replacing $\bx$ by $\bx^\beta$, where
$\beta$ is the prefix of $\alpha$ that has been read so far.

The \cspda attempts to establish that the input word $\alpha$ does not fix
some string in $\Omega$ with prefix $\bx$, which would prove that it lies in
$\coWP(G,A)$. To do that, it attempts to compute the image $\bx^\alpha$
of $\bx$ under $\alpha$, and then to check whether $\bx^\alpha$ is distinct from 
$\bx$, which would prove that $\alpha \in \coWP(G,A)$.
If $\alpha$ does lie in $\coWP(G,A)$, then there exists a string in
$\Omega$ that it does not fix and, provided that a sufficiently long prefix
$\bx$ of this string together with enough padding symbols are put on the check
stack in the first stage, the \cspda will succeed in showing that
$w \notin \coWP(G,A)$.

After reading the first input letter and copying $\bx$ onto the pushdown stack,
and also after reading all subsequent input letters $x = (B,C,\varphi) \in A$
and reconstituting the pushdown stack if necessary as explained below,
the \cspda looks for a prefix of the word $\bx^\beta$ that is currently on the
pushdown stack that lies in the set $B$. It can do this by examining the
prefixes of $\bx^\beta$ of length at most $K$. If $|\bx^\beta| < K$
then it could fail to find such a prefix, in which case the computation fails.
This means that the word $\bx$ was not long enough.
Otherwise a unique prefix $b \in B$ is identified, and this is replaced
by $b^\varphi$ on the top of the pushdown stack. If $b^\varphi$ is
longer than $b$, then this replacement might not be possible without going above
the top of the check stack, in which case the computation fails. This occurs
when not enough padding symbols were put at the top of the check stack.
Otherwise, after making the replacement, the read write head is located at the
top of the pushdown stack, and sees the first letter of the new word
$\bx^{\beta x}$.  Since the processed prefix $\beta$ of $\alpha$ has now been
replaced by $\beta x$, the word on the pushdown stack is now equal
to $\bx^{\beta x}$, as required.

As in the proof of Theorem~\ref{prop:bag}, before entering a reading state,
the \cspda checks whether the words $\bx^\beta$ and $\bx$ on the stacks are
equal. If they have different lengths or, equivalently, if the letter that
the read write head sees on the check stack does not lie in the set $Q$,
then the words cannot be equal, and we cannot have $\beta =_G 1$.
To see this note that, for example, the infinite string
$\bx\sigma_1\sigma_2\sigma_1\sigma_2^2\sigma_1\sigma_2^3\cdots$
is not fixed by $\beta$.
In this case the \cspda moves immediately into the accepting state $q_2$.

Otherwise, as in the proof of Theorem~\ref{prop:bag},
the \cspda checks the words for equality by emptying the pushdown stack of
letters that are the same as the corresponding letters on the check stack.
It then goes into state $q_1$ 
if the pushdown stack contains only $\bs$, and into $q_2$ otherwise. 
After reading the next input letter if any, if the pushdown stack 
contains only $\bs$,
or if the symbol currently at its top  does not lie in $Q$,
then the \cspda reconstitutes the word $\bx^\beta$ on the pushdown stack by
copying letters from the check stack.

Checking that the required properties of a \cspda hold for the automaton that
we have defined is straightforward, and similar to the proof in
Theorem~\ref{prop:bag}. 
\end{proofof}

\section{Closure properties of \sce groups}
\label{sec:clprops}
Our proofs of the closure properties in this section closely follow the proofs
of corresponding properties for stack groups in~\cite[Sections~7 and~8]{HR}.

\begin{theorem}\label{prop:clfin}
The property of being a \sce group is independent of the choice of finite
generating set. Furthermore,
the class of \sce groups is closed under passing to finitely generated
subgroups, finite direct products, and finite extensions.
\end{theorem}
\begin{proof}
Let $G=\langle A \rangle$ be a \sce group and $H = \langle B \rangle \le G$ with $B$ finite
and closed under inversion. For each $b \in B$, we let $w_b$ be a specific word
in $A^*$ with $b =_G w_b$. The special \cspda $\cM_H$ accepting $\coWP(H,B)$
is based on the special \cspda $\cM_G$ for $\coWP(G,A)$ but with input alphabet
$B$ and with some extra states and non-reading moves. On reading a letter in
$B$ in its deterministic second phase, the special \cspda $\cM_H$ behaves in
the same way as $\cM_G$ behaves on reading the word $w_B$. Note that by
choosing $H=G$ this also proves the final assertion in the theorem.

Now let $G = H \times K$ with $H$ and $K$ \sce groups, where $\coWP(H,A_H)$ and
$\coWP(K,A_K)$ are accepted by special \cspda $\cM_H$ and $\cM_K$.
We construct a special \cspda $\cM_G$ accepting $\coWP(G,A_H \cup A_K)$ as
follows. It starts with an $\epsilon$-move into the initial state of either
$\cM_H$ or $\cM_K$, and thereafter behaves like $\cM_H$ or $\cM_K$ respectively.
In the first case it ignores all input letters from $A_K$, and in the second
case it ignores those from $A_H$.

Now suppose that $H$ is a \sce group with $\coWP(H,B)$ accepted by the
special \cspda $\cM_H$, and that $H \le G$ with $|G:H|$ finite, and let $A$ be
an inverse-closed generating set of $G$ with $B \subseteq A$. Let $T$ be a
right transversal of $H$ in $G$ with $1_G \in T$. For each $t\in T$ and
$a \in A$, we have $Hta = Ht_a$ for some $t_a \in T$, and we let $w_{t,a}$ be
a fixed word in $B^*$ with $ta =_G w_{t,a}t_a$.

The special \cspda $\cM_G$ accepting $\coWP(G,A)$ is based on $\cM_H$, but its
states are pairs $(q,t)$, where $q$ is a state of $\cM_H$ and
$t \in T$. The initial state is $(q_0,1_G)$, where $q_0$ is the
initial state of $\cM_H$. The first phase of $\cM_G$ is the same as the first
phase of $\cM_H$, with the first component of the states changing as in
$\cM_H$ and the second component remaining $1_G$. In the second phase,
when $\cM_G$ reads the letter $a \in A$ in state $(q,t)$, it behaves as
though $\cM_H$ has read the word $w_{t,a}$ with the first state component
changing accordingly, and the second state component changes to $t_a$.
The \cspda $\cM_G$ accepts an input word $w$ if and only if, in the next reading
state $(q,t)$ after reading $w$, either $t \ne 1_G$ or $q$ is
an accept state of $\cM_H$.
\end{proof}

\begin{theorem}\label{prop:clwp}
The (standard) restricted wreath product $H \wr K$ of a special co-\E group
$H$ by a finitely generated virtually free top group $K$ is a \sce group.
\end{theorem}
\begin{proof} This is similar to the proof of the last part
of~\cite[Proposition~9]{HR}.
Let $H = \langle A_H \rangle$ be a \sce group and $\cM_H$ a special \cspda with
$\cL(\cM_H) = \coWP(H,A_H)$.
Let $K$ be finitely generated and virtually free. As in the proof of
Theorem~\ref{prop:vfree} above,
we can choose a generating set $A_K = B \cup T^{\pm 1}$ of $K$ in which
$B$ is the inverse closure of a free generating set of a free subgroup $F$ of
finite index in $K$ and $T$ is a right transversal of $F$ in $K$ with
the element in $F$ removed, and then every element of $K$ can be written
uniquely as a word $wt$ with $w$ a freely reduced word in $B^*$ and
$t \in T \cup \{\emptystring\}$. As we saw in the proof of
Theorem~\ref{prop:vfree}, the word problem $\WP(K,A_K)$ is the language of
a pushdown automaton $\cP_K$ that works by storing the normal form word $wt$
of the prefix of the input word read so far, and accepts when the stack is
empty.

Let $G$ be the restricted wreath product $H \wr K$.
Then $G = \langle A \rangle$ with $A = A_H \cup A_K$. (As usual, we
assume that $A_H$ and $A_K$ are disjoint.) We shall
describe a special \cspda $\cM_G$ with $\cL(\cM_G) = \coWP(G,A)$.
As in~\cite[Proposition~9]{HR} (and earlier results on closure under
restricted wreath product), it works by non-deterministically choosing a
component of the base group $H^K$ of $G$, and using $\cM_H$ to attempt to prove
the non-triviality of that component of an element of the base group.
A word will be accepted by $\cM_G$ if and only if either the input word does
not project onto the identity element of $K$, or $\cM_H$ accepts the word
representing the chosen component of the base group.

In this construction we set the bottom of stack symbol $\bs_G$ of $\cM_G$ to
be the bottom of stack symbol $\bs_H$ of $\cM_H$.
The initialisation of $\cM_G$ consists of an initialisation of $\cM_H$,
followed by the writing of an arbitrary normal form word $wt \in A_K^*$ on the
check stack (with $t$ on top), on top of the initialised stack for $\cM_H$.
An arbitrary number of padding symbols $\ps_G$ followed by
$\ts_G$ are then added to the check stack to help ensure that
Condition (P3) is satisfied.
This word $wt$ defines the component of the base group that we shall check
for non-triviality using the initialisation of $\cM_H$. The current state
of $\cM_H$ is stored at all times as a component of the state of $\cM_G$.

After writing $\bs_G$ at the bottom of the pushdown stack
and reading the first input symbol $x$, 
the read write head 
ascends to the top of the 
substack of the $\cM_G$ check stack that is 
the check stack of $\cM_H$ (while pushing padding
symbols $\ps_G$ to the pushdown stack) 
and then copies the word $wt$
from the check stack to the pushdown stack.

At this point, and also when processing subsequent input symbols $x$ when
the read write head is at the top of the pushdown stack, 
if $x \in A_K$, then $x$ is passed to $\cP_K$ for processing using its
pushdown stack (which is the upper part of the pushdown stack of $\cM_G$).
If $x \in A_H$ and the pushdown stack of $\cP_K$ is non-empty, then $x$ is
ignored by $\cM_G$.

If, however, the pushdown stack of $\cP_K$ is empty and an input symbol
$x \in A_H$ is read, then the read write head descends while removing padding
symbols $\ps_G$ from the pushdown stack, and $x$ is processed by $\cM_H$
using its current stored state. After $\cM_H$ has processed $x$, the read write
head moves back up (while putting symbols $\ps_G$ on the pushdown stack).

In any case, after processing the input letter as just described, before
entering another reading state, we first check whether the \cspda should be in
an entry configuration, which is the case if and only if
\begin{mylist}
\item[(i)] the normal form word for the element of $K$ stored at the top of the
pushdown stack is equal to $wt$; and
\item[(ii)] the \cspda $\cM_H$ is in an entry configuration.
\end{mylist}
To see whether the first of these conditions hold, we first check whether the
read write head is  
pointing at the top symbol $t$ of $wt$ (or to the
top letter in $B$ or to $\ts_H$ if $t = \emptystring$): if not, then
the two words have different lengths, and the condition does not hold.
If so, then $\cM_G$ checks the two words for equality by popping symbols
from the pushdown stack as long as they are equal to the corresponding symbols
on the check stack. If at any time the two symbols seen are not equal, then the
words are unequal, and any popped symbols are replaced on the pushdown stack. 
If the two words are equal, then we also check whether we are in an entry
configuration of $\cM_H$ by checking whether it is in an entry state and, if so,
checking whether the pushdown stack of $\cM_H$ 
contains only padding symbols above $\bs_H$,
 by popping padding
symbols $\ps_G$. If we are not in an entry configuration of $\cM_H$,
then again we reconstitute all of the previous contents of the pushdown
stack of $\cM_G$. But if we are in an entry configuration of $\cM_H$, then
$\cM_G$ transitions into  the corresponding entry configuration of $\cM_G$.
This process is necessary for $\cM_G$ to ignore the word problem of $G$.

At this stage we are in a reading state, either in an entry configuration of
$\cM_G$, or with the read write head at the top of the pushdown stack of
$\cM_H$, and we read the next input symbol if any.
If we are in an entry configuration, then we proceed as we did for
the first input letter read. We are then in any case at the the top of the
pushdown stack, and we proceed as before.

Following (P7), the accept states of $\cM_G$ are all reading states other than
the entry states. So all words that do not project onto $1_K$ are accepted
together with all words for which the selected component in the base group has
been accepted by $\cM_H$ as a word in the coword problem of $H$.
It is now straightforward to verify that the \cspda that we have described is
special. Property (P6) follows from the assumption that $\cM_H$ has this
property.
\end{proof}

\begin{theorem}\label{prop:clfp}
The class of \sce groups is closed under taking free products with finitely
many free factors.
\end{theorem}
\begin{proof} It is enough to prove that the free product $G = H*K$ of
two \sce groups $H$ and $K$ is a \sce group. Let $\cM_H$ and $\cM_K$ be
special \cspda accepting $\coWP(H,A_H)$ and $\coWP(K,A_K)$ respectively.
We shall define a special \cspda $\cM_G$ accepting $\coWP(G,A)$ with
$A = A_H \cup A_K$ (which we assume to be a disjoint union).
The \cspda $\cM_G$ includes disjoint copies of $\cM_H$ and $\cM_K$.
Beyond the single initial state of $\cM_G$, each other state of $\cM_G$ is 
naturally associated with a state of either $\cM_H$ or $\cM_K$.
However, since $\cM_G$ needs to do some additional
processing while remembering its current state of $\cM_H$ or $\cM_K$,
that current state will be stored only as the first component of a state of
$\cM_G$, which will have a second component containing further information.
We denote by $\Delta_H$ and $E_H$ the stack alphabet and set of entry states for
$\cM_H$, and by $\Delta_K$ and $E_K$ the stack alphabet and set of entry states
for $\cM_K$.

In the first phase (i.e.\ the initialisation) of $\cM_G$ an arbitrary string
$\alpha_1\beta_1,\cdots \alpha_k\beta_k$ for some $k \geq 1$ is put onto the
check stack (with $\alpha_1$ at the bottom), with substrings $\alpha_i$ and
$\beta_i$ over $\Delta_H \cup E_H$ and $\Delta_K \cup E_K$ respectively,
which are all non-empty for $i=1,\ldots,k$, except possibly for
$\alpha_1$ and/or $\beta_k$ (since the string could start or end with a
substring over either $\Delta_H \cup E_H$ or $\Delta_K \cup E_K$).
With one exception, each non-empty string $\alpha_i$ or $\beta_i$ consists of
a string over $\Delta_H$ or $\Delta_K$ that is valid for the initialisation of
$\cM_H$ or $\cM_K$ followed by the entry state of $\cM_H$ or $\cM_K$ that
results from that initialisation. The exception is that, for the first
non-empty string (which may be $\alpha_1$ or $\beta_1$) at the bottom of the
check stack, we put the bottom of stack marker $\bs_G$ in place of
$\bs_H$ or $\bs_K$ at the bottom. This ensures that Property (P3)
is satisfied by $\cM_G$.

Ignoring the above exception and the symbols for the entry states of the
components, we see that the check stack of $\cM_G$ that we create during this
initialisation is a sequence of \emph{check substacks}
that are initialised stacks of $\cM_H$ and $\cM_K$.
We let $q_\ent$ be the entry state of $\cM_H$ or $\cM_K$ that is at the top of 
the lowest of those check substacks (corresponding to the first non-empty
substack).
After the check stack has been set, the read write head of $\cM_G$ descends to
the bottom of the pair of stacks, the symbol $\bs_G$ is written on the
pushdown stack, and the entry state of $\cM_G$ is set to $(q_\ent,q_\ent)$.

At all times during the second phase, exactly one of the subgroups $H$ and $K$
will be the \emph{current subgroup}. When the current subgroup is $H$, the
\cspda $\cM_G$ will have a state of $\cM_H$ as its first component,
and its read write head will be scanning a substack that is a check stack of $\cM_H$,
which we call the \emph{current check stack}, and similarly for $K$.
The general idea is that, when the current subgroup is $H$
and input symbols from $A_H$ are being read, the \cspda behaves in the same
way as $\cM_H$, using the current check substack. We call such a
sequence of moves of $\cM_G$ a \emph{sub-computation} in $\cM_H$ or in $\cM_K$.
As we shall see later, the \cspda needs to remember the entry state of the
current sub-computation, and we can store this as the second component
of the current state. At the start of the second stage, the current subgroup is
$H$ or $K$, depending on whether the first check substack contains $\alpha_1$
or $\beta_1$.

We need to describe how $\cM_G$ behaves when it reads a letter $x$ from the
input tape. Suppose first that the current subgroup is $H$. 
Now suppose also that $x \in A_K$, so $x$ is not in the current subgroup.
In this case, $\cM_G$ moves to the top
of the current check substack while pushing padding symbols $\ps_G$ onto
the pushdown stack. At the top of this check substack, it records the current
state of $\cM_H$ on the pushdown stack (in the same position as the current
entry state is stored on the check substack). If there are no further
check substacks above the current one, then the initialisation stage of
$\cM_G$ was too short and $\cM_G$ fails; we say that the initialisation was
\emph{inadequate}.  Otherwise, the read write head moves up into the next check
substack, which will be for the group $K$. The read write head first moves
to the top of this check substack (while pushing padding symbols $\ps_G$
onto the pushdown stack) to retrieve the entry state $q_\ent$ of this
initialisation of $\cM_K$. Then it descends again to the bottom of this check
substack (popping from the pushdown stack as it does so), pushes $\bs_K$
onto the pushdown stack, and then behaves as though $\cM_K$
had read the input symbol $x$ while in entry state $q_\ent$. It changes to the
state $(q,q_\ent)$, where $q$ is the state into which a sub-computation on
$\cM_K$ would transition from $q_\ent$, given input $x$ and the current stack
symbols, and pushes to the pushdown stack, accordingly.

Otherwise we have $x \in A_H$. As explained earlier, $\cM_G$ now behaves in the
same way as $\cM_H$, using the currently scanned check substack as
its check stack. Where $q_\ent$ is the entry state of this sub-computation,
and $q$ is the state that $\cM_H$ should enter after reading $x$,
$\cM_G$ will now enter the state $(q,q_\ent)$.

So far we have assumed that the current subgroup is $H$.
If the current subgroup is $K$ then the machine operates in just the same way,
but with the roles of $H$ and $K$ in the description above reversed. 

Suppose that, after completing a reading move and any subsequent non-reading
moves of a sub-computation in the subgroup $H$, the \cspda is in the entry
configuration of the current sub-computation within $\cM_H$ but not in the
entry configuration of $\cM_G$. Note that this situation arises if and only if
the current state is $(q_\ent,q_\ent)$, where $q_\ent$ is the entry state of
the current computation, and the current symbol seen on the check stack is the
bottom of stack marker $\bs_H$ (and not $\bs_G$).
Then the current subgroup changes from $H$ to $K$, and the read write head
descends into the check substack immediately below its current position, which
will be for the group $K$. This lower check substack has recorded the
configuration of a sub-computation within $\cM_K$ that had been interrupted
when a letter from $A_H$ was read. The read write head retrieves the entry
state and the state of that sub-computation at the point of interruption from
the tops of the check substack and pushdown stack, respectively, and then it
descends and removes any padding symbols $\ps_G$ from the top of the
pushdown stack.  After doing that, the configurations of both $\cM_G$ and of
the sub-computation in $\cM_K$ are the same as they were at the point of
interruption. Note that, if the next input letter is from $A_H$ then $\cM_G$
immediately goes up again and resumes the current sub-computation in $\cM_H$,
but this temporary change of current subgroup and descent into the lower check
substack is necessary to ensure that $\cM_G$ ignores the word problem of $G$.

Reversing the roles of $H$ and $K$ in the paragraph above deals with the case
where the \cspda enters an entry configuration of the current sub-computation
within $\cM_K$ but is not in an entry configuration of $\cM_G$.
Following (P7), the accepting states of $\cM_G$ are all reading states other
than the entry states. That completes the description of the \cspda $\cM_G$. 

We shall show next that, so long as $\cM_H$ and $\cM_K$ are special then so is
$\cM_G$. It should be clear that $\cM_G$ satisfies the additional properties
(P1)--(P5) and (P7). It remains to verify Property (P6), that $\cM_G$ ignores
the word problem of $G$. So let $C$ be an entry configuration of $\cM_G$ and
let $uv \in A^*$ with $v =_G 1$, and assume also that $C^{uv}$ does not fail.
We need to show that $C^u=C^{uv}$.

Suppose first that the letters of $v$ are all in the same alphabet $A_H$ or
$A_K$ - we may assume without loss that $v \in  A_H^*$. Then $v =_H 1$, and
the claim follows essentially from the assumption that $\cM_H$ ignores the word
problem of $H$, because $\cM_G$ is carrying out a sub-computation in $\cM_H$
while reading $v$. There are some complications however. 

If the current subgroup of the configuration $C^u$ is $K$ then, on reading the
first letter of $v$, the current subgroup changes to $H$. Then, after reading
$v$, $\cM_H$ will be in an entry configuration of $\cM_H$, and the read write
head will descend into the check substack of $\cM_K$, retrieving the reading
state $M_K$ was left at and possibly removing some padding symbols,
resulting in $C^{uv}=C^u$.

It is also possible that, after reading a letter $x$ of $v$ that is not the
final letter of $v$, the current sub-computation within $\cM_H$ is in an entry
configuration of $\cM_H$, in which case (as we have explained above) the stack
read write head of $\cM_G$ descends into the check substack below. Since $x$
is not the last letter of $v$, then on reading the following letter of $v$,
the read write head moves back up and the sub-computation resumes in $\cM_H$.

If $v$ contains letters from both $A_H$ and $A_K$, then there must be
some non-trivial subword $v'$ of $v$ that contains letters only from
$A_H$ or $A_K$ with $v' =_H 1$ or $v' =_K 1$ and then the claim follows from
the previous paragraphs and induction on the number of \emph{components} of $v$
(where a component is a maximal non-empty subword of $v$ within either
$A_H^*$ or $A_K^*$).

It remains to prove that $\cL(\cM_G) = \coWP(G,A)$. 
Since the entry states of $\cM_G$ are not accepting states, and $\cM_G$ ignores
the word problem of $G$, no word $w \in A^*$ with $w =_G 1$ can be
accepted by $\cM_G$, and so we have $\cL(\cM_G) \subseteq \coWP(G,A)$.

Conversely, suppose that $w \in \coWP(G,A)$, and write $w = w_1 \cdots  w_k$
as a sequence of components. We shall prove by induction on the number $k$ of
components that $w \in \cL(\cM_G)$. In fact we shall use
Lemma~\ref{lem:nonfail} to prove the stronger statement that, for any
$n \in \Z$ with $n>0$, there is an entry configuration $C$ of $\cM_G$ such that
\begin{mylist}
\item[(i)] $C^w$ is an accept configuration of $\cM_G$; and
\item[(ii)] for each entry configuration $C_i$ of $\cM_H$ or $\cM_K$ that is
defined by one of the check substacks in the initialisation of $\cM_G$ that
results in $C$, and any word $v$ over $A_H$ or $A_K$ of length at most $n$,
$C_i^v$ does not fail.
\end{mylist}

If $k=1$ then $w$ is a word over $A_H$ or $A_K$ in either $\coWP(H,A)$ or
$\coWP(K,A)$, and so $\cM_G$ operates as either $\cM_H$ or $\cM_K$,
whose languages are both within $\cL(\cM_G)$. The stronger statement in
the preceding paragraph follows from Lemma~\ref{lem:nonfail} applied to
$\cM_H$ or $\cM_K$. So suppose that $k>1$.

If each component of $w$ lies in either $\coWP(H,A_H)$ or $\coWP(K,A_K)$,
then each lies in either $\cL(\cM_H)$ or $\cL(\cM_K)$, and there are
entry configurations of $\cM_H$ and of $\cM_K$ that result in these components
being accepted and also, by Lemma~\ref{lem:nonfail}, do not fail on input of
any word over $A_H$ or $A_K$ of length at most $n$.
By initialising $\cM_G$ with the corresponding sequence of initialisations of
$\cM_H$ and of $\cM_K$, we ensure that each component of $w$ is processed by a
sub-computation in a separate substack and, after reading $w$, the \cspda
$\cM_G$ is in the accept state arising from the sub-computation of $w_k$.
So $w \in \cL(\cM_G)$ with the conditions of the stronger assertion satisfied. 

Otherwise choose $i$ such that the component $w_i$ of $w$ lies within the word
problem of one of $H$ or $K$, and define
$u\coloneqq w_1\cdots w_{i-1}$, $u'\coloneqq w_{i+1}\cdots w_k$, and
$v\coloneqq w_i$; if $i=1$ or $i=k$ then $u$ or $u'$ is defined to be empty.
Since, by assumption, we have $w = uvu' \in \coWP(G,A)$ and $v =_G 1$, we have
$uu' \in \coWP(G,A)$ and $v \in \WP(G,A)$. By inductive hypothesis,
$uu'$, which has at most $k-1$ components, is within $\cL(\cM_G)$,
so there is an entry configuration $C$ of $\cM_G$ for which $C^{uu'}$ is an
accept configuration. Furthermore, for any $n>0$, we can choose $C$ to
satisfy Condition (ii) above, and in any case we choose it to satisfy this
condition with $n$ equal to the length of $w$.

We claim that we can choose $C$ such that $C^{uv}$ does not fail and such that
$C$ satisfies Condition (ii) above. By Condition (ii) applied to our
initial choice of $C$ such that $C^{uu'}$ is an accept state, the only way that
$C^{uv}$ could fail would be if $C$ was an inadequate initialisation for reading
$C^{uv}$, in which case we would require a single extra check substack to
read the subword $v$ of $uv$. So we simply replace $C$ by the same
configuration but with the required additional substack on top that satisfies
Condition (ii) for the chosen $n$ assumed to be at least equal to the length of
$w$. (We can assume that $H$ and $K$ are both non-trivial, and so $\cL(\cM_H)$
and $\cL(\cM_K)$ are non-empty, and we can apply Lemma~\ref{lem:nonfail} to
choose this new substack.)

Now, since $\cM_G$ ignores the word problem (i.e.\ we have Property (P6))
and $C^{uv}$ does not fail, we have $C^w = C^{uvu'} = C^{uu'}$ is an accept
configuration, and $C$ satisfies Condition (ii) above as required.
\end{proof}

Raad Al Kohli,\\
Mathematics Institute, University of Warwick, Coventry CV4 7AL,\\
{\tt Raad.Al-Kohli@warwick.ac.uk}

Derek F Holt, \\
Mathematics Institute, University of Warwick, Coventry CV4 7AL,\\
{\tt derekholt127@gmail.com}

Sarah Rees, \\
School of Mathematics, Statistics and Physics, Newcastle University, Newcastle NE1 7RU,
{\tt Sarah.Rees@newcastle.ac.uk}
\end{document}